\documentclass[11pt,leqno]{article}
\usepackage{a4wide}

\usepackage{url}

\usepackage{times}
 
\usepackage{amsmath}
\usepackage{amsthm}
\usepackage{amsfonts}
\usepackage{amssymb}

\usepackage{enumerate}

\usepackage[utf8]{inputenc}
\usepackage[OT2,T1]{fontenc}
\DeclareSymbolFont{cyrletters}{OT2}{wncyr}{m}{n}
\DeclareMathSymbol{\Sha}{\mathalpha}{cyrletters}{"58}

\usepackage{comment}

\newtheorem{thm}{Theorem}[section]
\newtheorem{lemma}[thm]{Lemma}

\newtheorem{cor}[thm]{Corollary}

\usepackage{hyperref}

\usepackage{amscd}

\usepackage[mathscr]{eucal}
\usepackage[british]{babel}

\usepackage{indentfirst}

\usepackage{footmisc}

\setfnsymbol{wiley}

\newcommand{\Selp}{\mathrm{Sel}_{p^\infty}}
\newcommand{\Sel}{\mathrm{Sel}}
\newcommand{\QQ}{\mathbb{Q}}

\newcommand{\Kc}{K^{\mathrm{cyc}}}
\newcommand{\Lc}{L^{\mathrm{cyc}}}
\newcommand{\vc}{{v_{\mathrm{cyc}}}}
\newcommand{\vcc}{{v_\infty}}

\newcommand{\vlc}{{v_{L^{\mathrm{cyc}}}}}
\newcommand{\Kcc}{K_\infty}

\newcommand{\Fp}{\mathbb{F}_p}

\newcommand{\MHG}{\mathfrak{M}_H(G)}
\newcommand{\NHG}{\mathfrak{N}_H(G)}

\newcommand{\ZZ}{\mathbb{Z}}

\newcommand{\Zp}{\mathbb{Z}_p}
\newcommand{\Qp}{\mathbb{Q}_p}
\newcommand{\Gal}{\operatorname{Gal}}
\newcommand{\rk}{\mathrm{rk}}

\newcommand{\GL}{\mathrm{GL}}

\newcommand{\Ker}{\mathrm{Ker}}
\newcommand{\Tor}{\mathrm{Tor}}
\newcommand{\X}{X(E/K_{\infty})}

\newcommand{\Ext}{\mathrm{Ext}}

\newcommand{\Hom}{\mathrm{Hom}}

\newcommand{\Lie}{\mathrm{Lie}}

\newcommand{\vi}{\mathrm{v}}

\newcommand{\Coker}{\mathrm{Coker}}

\newcommand{\bs}{[\hspace{-0.04cm}[}
\newcommand{\js}{]\hspace{-0.04cm}]}

\newtheorem{pro}[thm]{Proposition}
\newtheorem{lem}[thm]{Lemma}
\newtheorem{con}[thm]{Conjecture}
\newtheorem{hyp}{Hypothesis}

\newtheorem{df}[thm]{Definition}

\newtheorem*{example}{Example}
\newtheorem*{remark}{Remark}

\newif\ifapx
\newif\ifkron
\newif\ifexample

\title{Algebraic functional equations and completely faithful Selmer groups}
\author{T.\ Backhausz and G.\ Z\'abr\'adi\footnote{The second author was partially supported by OTKA Research grant no.\ K-101291.}}

\begin{document}
\maketitle

\begin{abstract}
Let $E$ be an elliptic curve---defined over a number field $K$---without complex multiplication and with good ordinary reduction at all the primes above a rational prime $p\geq 5$. We construct a pairing on the dual $p^\infty$-Selmer group of $E$ over any strongly admissible $p$-adic Lie extension $\Kcc/K$ under the assumption that it is a torsion module over the Iwasawa algebra of the Galois group $G=\Gal(\Kcc/K)$. Under some mild additional hypotheses this gives an algebraic functional equation of the conjectured $p$-adic $L$-function. As an application we construct completely faithful Selmer groups in case the $p$-adic Lie extension is obtained by adjoining the $p$-power division points of another non-CM elliptic curve $A$.
\end{abstract}

\section{Introduction}



The main conjectures of Iwasawa theory usually state that $(i)$ there
exists a $p$-adic $L$-function attached to the elliptic curve over
a $p$-adic Lie extension of $\mathbb{Q}$ which interpolates the
special values of the complex $L$-functions of the elliptic curve
twisted by Artin representations of the Galois group, and $(ii)$, this $p$-adic $L$-function
is a characteristic element for the dual of the Selmer group.
These are the only tools known at present for studying the
mysterious relationship between the arithmetic
properties of elliptic curves and the special values of their
complex $L$-functions, especially for attacking the conjectures of
Birch and Swinnerton-Dyer. In the noncommutative setting of \cite{CFKSV} the $p$-adic $L$-function lies in the
algebraic $K_1$-group of $\Lambda(G)_{S^*}$, the Iwasawa algebra
of the Galois group localized by a canonical Ore set.

Let $E$ be an elliptic curve without complex multiplication defined over a number field $K$ and $p\geq 5$ be a prime such that $E$ has good ordinary reduction at all the primes above $p$ in $K$. The aim of this paper is to prove an algebraic functional equation for the conjectured $p$-adic $L$-function over arbitrary strongly admissible $p$-adic Lie extensions $\Kcc$. This generalizes earlier results of the second author in case $K=\mathbb{Q}$ and $\Kcc$ being the false Tate curve extension \cite{ZfT} and the $\GL_2$-extension associated to $E$ \cite{Z}.

This ``algebraic functional equation'' is on the level of the class of dual Selmer groups in the Grothendieck group $K_0(\MHG)$ of a certain category $\MHG$ of (finitely generated left) $\Lambda(G)$-modules where $\Lambda(G)$ is the Iwasawa algebra of the Galois group $G=\Gal(\Kcc/K)$. Here the objects of the category $\MHG$ are those $\Lambda(G)$ modules $M$ for which $M/M(p)$ is finitely generated over the smaller Iwasawa algebra $\Lambda(H)$ where $H=\Gal(\Kcc/\Kc)$. Let us denote by $\#$ the anti-involution on $\Lambda(G)$ induced by the inversion map $\#\colon G\ni g\mapsto g^{-1}\in G$. For an left $\Lambda(G)$-module $M$ we denote by $M^\#$ the right $\Lambda(G)$-module on the same underlying abelian group as $M$ obtained by the rule $m\lambda:=\lambda^\# m$ for $m\in M$ and $\lambda\in\Lambda(G)$. Now if $M$ lies in the category $\MHG$ then 
\begin{eqnarray*}
\#\colon K_0(\MHG)&\to& K_0(\MHG)\\ 
{[M]}&\mapsto& \sum_{i=1}^{\dim G+1}[\Ext^i_{\Lambda(G)}(M^{\#},\Lambda(G))]
\end{eqnarray*}
is an involution on the Grothendieck group $K_0(\MHG)$. Here we regard $\Ext^i_{\Lambda(G)}(M^{\#},\Lambda(G))$ as a left module over $\Lambda(G)$ via the left multiplication of $\Lambda(G)$ on itself. By an algebraic functional equation for a module $M$ we mean a relation involving the classes $[M]$ and $\#([M])$.

In the classical case when $\Kcc$ is the cyclotomic $\Zp$-extension of a number field $L$ Perrin-Riou \cite{P} constructed a $\Lambda(\Gamma_L)$-homomorphism $X(E/\Lc)\to\Ext^1_{\Lambda(\Gamma_L)}(X(E/\Lc)^{\#},\Lambda(\Gamma_L))$ as the projective limit of the Cassels-Tate-Flach pairing (here $\Gamma_L=\Gal(\Lc/L)$). By taking projective limit of these maps for varying $K\leq L\leq \Kcc$ we obtain a map $$\varphi : X(E/\Kcc) \to \Ext^1(X(E/\Kcc)^\#,\Lambda(G))\ .$$ Our first Theorem (Thm.\ \ref{altalanostetel}) is a careful analysis of the kernel and cokernel of this map. We show that the kernel of $\varphi$ has trivial class in the Grothendieck group $K_0(\MHG)$. On the other hand, the cokernel equals---up to error terms with vanishing class in $K_0(\MHG)$---the direct sum of certain local factors $\Lambda(G)\otimes_{\Lambda(G_{v_\infty})}T_p(E)^*$ for primes $v\nmid p$ in $K$ ramifying infinitely in $\Kcc$ such that the Tate module $T_p(E)$ is defined over the completion $K_{\infty,\vcc}$ of $K_\infty$ at a (once and for all) fixed prime $v_\infty$ in $K_\infty$ above $v$. We denote by $P_1$ (resp.\ $P_2$) the set of those primes $v$ having the above property that are potentially multiplicative (resp.\ potentially good) for the curve $E$. 

In order to obtain a functional equation for $\X$ we also have to deal with the higher extension groups $\Ext^i_{\Lambda(G)}(\X^{\#},\Lambda(G))$ ($i\geq 2$). In section \ref{higherext} we show that the class of these extension groups vanishes in $K_0(\MHG)$ under certain local conditions at primes above $p$ only. Moreover, if $G$ is isomorphic to an open subgroup of $\GL_2(\Zp)$ then we do not need to assume these local conditions. So under these assumptions we obtain a functional equation of the characteristic element of $\X$---the conjectured $p$-adic $L$-function---in the $K_1$ of the localized Iwasawa algebra.

In section \ref{compatible} we show the compatibility of this functional equation with the conjectured interpolation property of the $p$-adic $L$-function. Note that the local factor $\alpha_v\in K_1(\Lambda(G)_S)$ appearing in the functional equation is the characteristic element of $T_p(E)^*$ as a module over the local Iwasawa algebra $\Lambda(G_{\vcc})$. When one substitutes an Artin representation $\rho$ of $G$ into $\alpha_v$ the value we get is the quotient $\frac{L_v(E,\rho^*,1)}{L_v(E,\rho,1)}$ of local $L$-values up to an epsilon factor.  Here the difficulty is that while in the interpolation property all the local $L$-factors at primes ramifying infinitely in $\Kcc$ are removed, in the algebraic functional equation only those show up for which the Tate module $T_p(E)$ is defined over the completed field $K_{\infty,\vcc}$. However, one can resolve this seeming contradiction can be resolved easily if one passes to a finite prime-to-$p$ extension $F_\infty$ of $\Kcc$ over which $T_p(E)$ is defined after completion at a prime $w_\infty$ dividing $\vcc$. In the functional equation of the characteristic element of $X(E/F_\infty)$ the local factors $\alpha_v$ do appear, however these elements map to $0$ under the composite map $$K_1(\Lambda(\Gal(F_\infty/K))_{S_F^*})\overset{\pi}{\to} K_1(\Lambda(G)_{S^*})\overset{\partial_G}{\to}K_0(\MHG)$$
therefore they do not appear in the functional equation over $\Kcc$ even though $\pi(\alpha_v)$ still interpolates quotients of local $L$-values at $v$ of twists of $E$ by Artin representations factoring through $G$.

The rest of the paper is devoted to the case when $G\cong H\times Z$ is a pro-$p$ group such that the Lie algebra of $H$ is split semisimple over $\mathbb{Q}_p$ and $Z=Z(G)\cong\Zp$ is the centre of $G$. For instance the group $\GL_2(\Zp)$ has a system of open neighbourhoods of the identity consisting of subgroups $G$ having this property. In this case we show that whenever the set of primes $P_1\cup P_2$ is nonempty then the dual Selmer $\X$ is not annihilated by any element in the centre of $\Lambda(G)$. Therefore by a result of Ardakov \cite{A} its global annihilator in $\Lambda(G)$ also vanishes. Moreover, if $\X$ is further assumed to have $\Lambda(H)$-rank $1$ then its image $q(\X)$ in the quotient category by the full subcategory of pseudonull modules is completely faithful (that is all the subquotients of $\X$ are pseudonull or have trivial global annihilator). Finally we give an example of two curves $E$ and $A$ such that in case $\Kcc=\mathbb{Q}(A[5^{\infty}])$ all these assumptions are satisfied therefore $q(\X)$ is completely faithful. As previous work of the first author \cite{B} shows we cannot take $E=A$ since in this case $\X$ always has $\Lambda(H)$-rank bigger than $1$.

\subsection{Notation}\label{not}

$\Kcc$ will always mean a strongly admissible extension of a number field $K$, in the sense that
\begin{enumerate}
\item $\Kcc$ contains $K(\mu_{p^\infty})$.
\item $\Kcc/K$ is unramified outside a finite set of places of $K$.
\item $\Gal(\Kcc/K)$ is a $p$-adic Lie group of dimension at least $2$
\item $\Gal(\Kcc/K)$ has no elements of order $p$
\end{enumerate}
We make the following definitions.\\
\begin{tabular}{c c}
\begin{tabular}{r c p{5cm}}
$G_L$ & $=$ & $\Gal(\Kcc/L)$ \\
$H_L$ & $=$ &  $\Gal(\Kcc/\Lc)$ \\
$\Gamma_L$& $=$ & $\Gal(\Lc/L)$ \\
$\Gamma_L^\ast$& $=$ & $\Gal(\Lc/K)$ \\
~&~&~\\
$M^\vi$ & $=$ & $\Hom_\text{cont} (M,\Qp/\Zp)$\\
$M^\ast$ & $=$ & $\Hom_\text{cont}(M,\Zp)$\\
\end{tabular}
\begin{tabular}{r c p{5cm}}
$G$ & $=$ & $\Gal(\Kcc/K)$ \\
$H$ & $=$ &  $\Gal(\Kcc/\Kc)$ \\
$\Gamma$& $=$ & $\Gal(\Kc/K)$ \\
~&~&~\\
~&~&~\\
$M^\#$ & $=$ & opposite module of $M$\\
$a_R^i(M)$ & $=$ & $\Ext^i_R(M,R)$\\
\end{tabular}
\end{tabular}\\

Moreover, we are going to denote by $v$ the primes in $K$ and by $q_v:=N_{K/Q}(v)$ their absolute norm. Further, $\vc,v_L,\vlc,\vcc$ will denote primes above $v$ in $\Kc$, $L$, $\Lc$, and $\Kcc$, respectively. Moreover, if $F_1\leq F_2$ is a Galois extension and $w$ is a prime in the field $F_2$ then we denote by $\Gal(F_2/F_1)_w$ the decomposition subgroup of the Galois group $\Gal(F_2/F_1)$.

Let $\mathcal{G}$ be any $p$-adic Lie group without elements of order $p$ and with a closed normal subgroup $\mathcal{H}\lhd \mathcal{G}$ such that $\Gamma:=\mathcal{G}/\mathcal{H}\cong\mathbb{Z}_p$. We are going to need the special case when $\mathcal{G}$ is a finite index subgroup of $\Gal(\mathbb{Q}(E[p^\infty])/\mathbb{Q})$ and also in the case when $\mathcal{G}\cong\mathbb{Z}_p$. The former embeds into $\GL_2(\mathbb{Z}_p)$ once we choose a $\mathbb{Z}_p$-basis of $T_p(E)$. We denote by $\Lambda(\mathcal{G})$ the Iwasawa $\mathbb{Z}_p$-algebra of $\mathcal{G}$ and by $\Omega(\mathcal{G})$ its $\mathbb{F}_p$-version.

Let $S$ be the set of all $f$ in $\Lambda(\mathcal{G})$ such that $\Lambda(\mathcal{G})/\Lambda(\mathcal{G})f$ is a finitely generated $\Lambda(\mathcal{H})$-module and
\begin{equation*}
S^*=\bigcup_{n\geq 0}p^nS.
\end{equation*}
These are multiplicatively closed (left and right) Ore sets of $\Lambda(\mathcal{G})$ \cite{CFKSV}, so we can define $\Lambda(\mathcal{G})_S$, $\Lambda(\mathcal{G})_{S^*}$ as the localizations of $\Lambda(\mathcal{G})$ at $S$ and $S^*$. We write $\mathfrak{M}_{\mathcal{H}}(\mathcal{G})$ for the category of all finitely generated $\Lambda(\mathcal{G})$-modules, which are $S^*$-torsion. A finitely generated left module $M$ is in $\mathfrak{M}_{\mathcal{H}}(\mathcal{G})$ if and only if $M/M(p)$ is finitely generated over $\Lambda(\mathcal{H})$ \cite{CFKSV}. We write $K_0(\mathfrak{M}_{\mathcal{H}}(\mathcal{G}))$ for the Grothendieck group of the category $\mathfrak{M}_{\mathcal{H}}(\mathcal{G})$. Similarly, let $\mathfrak{M}(\mathcal{G},p)$ denote the category of $p$-power-torsion finitely generated $\Lambda(\mathcal{G})$-modules and $\mathfrak{N}_{\mathcal{H}}(\mathcal{G})$ the category of $\Lambda(\mathcal{G})$-modules that are finitely generated over $\Lambda(\mathcal{H})$. 

\begin{lem}\label{ptors}
Assume in addition that $\mathcal{G}$ is a pro-$p$ group. Then we have $K_0(\mathfrak{M}_{\mathcal{H}}(\mathcal{G}))=K_0(\mathfrak{M}(\mathcal{G},p))\oplus K_0(\mathfrak{N}_{\mathcal{H}}(\mathcal{G}))$.
\end{lem}
\begin{proof}
By definition any module in $\mathfrak{M}_{\mathcal{H}}(\mathcal{G})$ is an extension of a module in $\mathfrak{M}(\mathcal{G},p)$ and a module in $\mathfrak{N}_{\mathcal{H}}(\mathcal{G})$. Hence we have $K_0(\mathfrak{M}_{\mathcal{H}}(\mathcal{G}))=K_0(\mathfrak{M}(\mathcal{G},p)) + K_0(\mathfrak{N}_{\mathcal{H}}(\mathcal{G}))$. Let $M$ and $N$ be $\Lambda(G)$-modules as above. Now we claim that the map $[M]\mapsto [M(p)]$ is well defined and extends to a homomorphism $K_0(\mathfrak{M}_{\mathcal{H}}(\mathcal{G}))\to K_0(\mathfrak{M}(\mathcal{G},p))$. For this let 
\begin{equation*}
0\to A\to B\to C\to 0
\end{equation*}
be a short exact sequence in $\mathfrak{M}_{\mathcal{H}}(\mathcal{G})$. Then we have $\mu(B)=\mu(A)+\mu(C)$ for their $\mu$-invariants (as $\Lambda(G)$-modules) since $p$-power-torsion $\Lambda(G)$-modules that are finitely generated over $\Lambda(H)$ (ie.\ modules in $\mathfrak{M}(\mathcal{G},p)\cap \mathfrak{N}_{\mathcal{H}}(\mathcal{G})$) clearly have trivial $\mu$-invariant. Here the $\mu$-invariant $\mu(M)$ of a finitely generated $\Lambda(G)$-module is defined by $\sum_{j=0}^{\infty}\rk_{\Omega(\mathcal{G})}(p^jM(p)/p^{j+1}M(p))$. Hence we have $[B(p)]=[A(p)]+[C(p)]$ in $K_0(\mathfrak{M}(\mathcal{G},p))$ by the main theorem of \cite{AW} applied to pro-$p$ groups. The statement follows using again that modules in $\mathfrak{N}_{\mathcal{H}}(\mathcal{G})$ have trivial $\mu$-invariant hence the homomorphism constructed is zero on $K_0(\mathfrak{N}_{\mathcal{H}}(\mathcal{G}))$.
\end{proof}

Further, if $M$ is a left $\Lambda(\mathcal{G})$-module, then by $M^{\#}$ we denote the right module defined on the same underlying set with the action of $\Lambda(\mathcal{G})$ via the anti-involution $\#=(\cdot)^{-1}$ on $\mathcal{G}$, i. e. for an $m$ element in $M$ and $g$ in $G$, and the right action is defined by $mg:=g^{-1}m$. By extending the right multiplication linearly to the whole Iwasawa algebra we get $mx=x^{\#}m$.

\subsection{Elliptic curves over strongly admissible extensions}
Let $E$ be an elliptic curve without complex multiplication. 
We are interested in the $\Lambda(G)$-module $X(E/\Kcc)=\Selp(E/\Kcc)^\vi$, which is the Pontryagin dual of the $p^\infty$-Selmer group of $E$ over $\Kcc$. This is the subject of the non-commutative main conjecture in the Iwasawa theory of elliptic curves \cite{CFKSV}. It is conjectured that $X(E/K_{\infty})$ always lies in $\MHG$ provided that $E$ has good ordinary reduction at $p$.  The following is a positive result in this direction.

\begin{df}
For a finite or infinite field extension $L \ge K$, let $R(L)$ be the set of primes in $L$ where the extension $\Kcc/L$ is infinitely ramified and $P_0(L)$ the subset of $R(L)$ consisting of primes not dividing $p$. For each prime $v\in L$ choose a prime $v_{\mathrm{cyc}}$ (resp.\ $v_\infty$) above $v$ in $\Lc$ (resp.\ $\Kcc$).
\begin{align*}
P_1(L)&=\left\{v \in P_0(L) | E \text{~has~split~multiplicative~reduction~at~}v_\infty\right\}\\
P_2(L)&=\left\{v \in P_0(L) | E \text{~has~good~reduction~at~} v_\infty \mathrm{~and~} E(\Lc_{\vc})[p] \neq 0\right\}
\end{align*}
$P_i(K)$ (resp.\ $R(K)$) will be denoted by $P_i$ (resp.\ $R$) for convenience.
\end{df}

\begin{thm}[{\cite[Theorem 5.4]{HS}}]\label{hachimoriformula}
Let $\Kcc/K$ be a strongly admissible pro-$p$ $p$-adic Lie extension. Let $E$ be an elliptic curves defined over $K$ that has good ordinary reduction at $p$. Assume that $X(E/\Kc)$ is finitely generated over $\Zp$. Then $X(E/\Kcc)$ is finitely generated over $\Lambda(H)$ (i.e. it is in $\NHG$) and 
$$\rk_{\Lambda(H)} X(E/\Kcc) = \rk_{\Zp} X(E/\Kc) + |P_1(\Kc)|+2|P_2(\Kc)|\ .$$
\end{thm}

\begin{remark}
The primes in $P_1\cup P_2$ are exactly those primes $v\in P_0$ for which all the $p$-power division points $E[p^\infty]$ are contained in the local tower extension $K_{\infty,\vcc}$ (Prop.\ 5.1 in \cite{HM}).
\end{remark}

\section{A pairing over strongly admissible extensions}

The following theorem is the key for establishing the functional equation of the characteristic element of $X(E/\Kcc)$.
\begin{thm}
\label{altalanostetel}
Let $E$ be an elliptic curve without complex multiplication having good ordinary reduction at all the primes above $p \ge 5$. 
Let $K$ be a number field and $\Kcc$ a strongly admissible extension. 
Assume that $X(E/\Kcc)$ has rank $0$ as a $\Lambda(G)$-module.

Then there is a map 
$$\varphi : X(E/\Kcc) \to \Ext^1(X(E/\Kcc)^\#,\Lambda(G))$$
with kernel representing $0$ and cokernel representing 
$$\bigoplus_{v \in P_1 \cup P_2}\Lambda(G) \otimes_{\Lambda(G_{\vcc})} T_p(E(K_{\infty,\vcc}))^\ast$$
in $K_0(\MHG)$.
\end{thm}

The following lemmas prepare the proof of this theorem.

\begin{lemma}
\label{novkorlat}
\begin{enumerate}[a)]
\item $H^i(H_L,E[p^{\infty}](\Kcc))$ has a bounded number of generators as $L \to \Kcc$.
\item Moreover, if $E[p^{\infty}](\Kcc)$ is finite then $|H^i(H_L,E[p^{\infty}](\Kcc))|$ is bounded as $L \to \Kcc$.
\end{enumerate}
\end{lemma}
\begin{proof}
To show a), it is enough to prove that $\dim_{\Fp} H^i(H_L,\Fp)$ is bounded as in \cite[Lemma 5.3]{Z}.

We may assume that $L$ is large enough so that $H_L$ acts trivially on $E[p^{\infty}](\Kcc)$.
Then, by induction on the order of $E[p^{\infty}](\Kcc)$, $|H^i(H_L,E[p^{\infty}](\Kcc))|$ is expressible as a function of $|H^j(H_L,\Fp)|$. 
It is therefore sufficient to bound these groups to prove b).

By a theorem of Lazard, each compact $p$-adic analytic group contains an open characteristic uniformly powerful subgroup \cite[Theorem 5.1.1]{SW}.
Therefore we may take the limit $L \to \Kcc$ with $H_L$ always a uniformly powerful subgroup, which is also
a pro-$p$ $p$-adic analytic group (these properties are inherited from $H_K$).
Then $\dim_{\Fp} H^i(H_L,\Fp)=\dim H_L$ and $\dim_{\Fp} H^i(H_L,\Fp)=\binom{\dim H_L}{i}$ by \cite[Theorem 5.1.5.]{SW}.
Because $H$ is a uniform pro-$p$ group, it has a \emph{finite rank} giving an upper bound to $\dim H_L$.
\end{proof}

\begin{lem}\label{veges}
Let $\mathcal{G}$ be a compact $p$-adic Lie group with a closed normal subgroup $\mathcal{H}\lhd \mathcal{G}$ such that $\Gamma=\mathcal{G}/\mathcal{H}\cong\mathbb{Z}_p$. Then any $\Lambda(\mathcal{G})$-module $M$ of finite order has trivial class in $K_0(\mathfrak{M}_\mathcal{H}(\mathcal{G}))$.
\end{lem}
\begin{proof}
As any finite module is a successive extension of $p$-torsion $\Lambda(\mathcal{G})$ modules we may assume without loss of generality that $M$ is a finite dimensional vectorspace over $\mathbb{F}_p$. Note that $\mathbb{F}_p$ with the trivial action of $\mathcal{G}$ has class $0$ in $K_0(\mathbb{F}_p[[\mathcal{G}]])$ and hence also in $K_0(\mathfrak{M}_\mathcal{H}(\mathcal{G}))$ as we have a short exact sequence
\begin{equation*}
0\to \mathbb{F}_p[[\Gamma]]\to \mathbb{F}_p[[\Gamma]]\to\mathbb{F}_p\to 0
\end{equation*}
induced by the augmentation map. On the other hand, if $M$ is an $\mathbb{F}_p[[\mathcal{G}]]$-module finite over $\mathbb{F}_p$ then for any finitely generated $\mathbb{F}_p[[\mathcal{G}]]$-module $N$ the module $M\otimes_{\mathbb{F}_p}N$ with the diagonal action of $\mathcal{G}$ is also finitely generated over $\mathbb{F}_p[[\mathcal{G}]]$. Therefore $M\otimes_{\mathbb{F}_p}\cdot$---being exact---induces a homomorphism $K_0(\mathbb{F}_p[[\mathcal{G}]])\to K_0(\mathbb{F}_p[[\mathcal{G}]])$ mapping $[0]=[\mathbb{F}_p]$ to $[M]$ whence we also have $[M]=[0]$.
\end{proof}

\begin{lem}\label{visszahuzas}
Let $\mathcal{G}_1$ and $\mathcal{G}_2$ be two compact $p$-adic Lie groups and $\psi\colon \mathcal{G}_1\to \mathcal{G}_2$ a continuous homomorphism with open image. Suppose that there is a closed normal subgroup $\mathcal{H}_i\lhd \mathcal{G}_i$ ($i=1,2$) such that $\mathcal{G}_i/\mathcal{H}_i\cong\mathbb{Z}_p$ and $\mathcal{H}_1=\psi^{-1}(\mathcal{H}_2)$. Assume further that the modules in the category $\mathfrak{M}_{\mathcal{H}_2}(\mathcal{G}_2)$ that are finitely generated over $\mathbb{Z}_p$ have trivial class in $K_0(\mathfrak{M}_{\mathcal{H}_2}(\mathcal{G}_2))$. Then the same is true for the group $\mathcal{G}_1$, ie.\ for any finitely generated $\mathbb{Z}_p$-module $M$ in the category $\mathfrak{M}_{\mathcal{H}_1}(\mathcal{G}_1)$ we have $[M]=0$ in $K_0(\mathfrak{M}_{\mathcal{H}_1}(\mathcal{G}_1))$.
\end{lem}
\begin{proof}
By assumptions $\psi(\mathcal{H}_1)$ has finite index in $\mathcal{H}_2$ therefore any finitely generated $\Lambda(\mathcal{H}_2)$-module is finitely generated over $\Lambda(\mathcal{H}_1)$ when regarded as a $\Lambda(\mathcal{H}_1)$-module via the map $\psi$. Therefore we obtain an exact functor
\begin{eqnarray*}
\psi^*\colon\mathfrak{M}_{\mathcal{H}_2}(\mathcal{G}_2)&\to&\mathfrak{M}_{\mathcal{H}_1}(\mathcal{G}_1)\\
M&\mapsto& M \text{ as a }\Lambda(\mathcal{G}_1)\text{-module via }\psi\ .
\end{eqnarray*}
So $\psi^*$ induces a homomorphism (still denoted by $\psi^*$ by a slight abuse of notation) on the $K_0$ having the property that $[\mathbb{Z}_p]=\psi^*([\mathbb{Z}_p])=\psi^*([0])=[0]$ where we regard $\mathbb{Z}_p$ a $\Lambda(\mathcal{G}_2)$ (resp.\ $\Lambda(\mathcal{G}_1)$) module with the trivial action of $\mathcal{G}_2$ (resp.\ of $\mathcal{G}_1$).

Now let $M\in\mathfrak{M}_{\mathcal{H}_1}(\mathcal{G}_1)$ be finitely generated over $\mathbb{Z}_p$. By Lemma \ref{veges} we may assume that $M$ has no $p$-torsion. Therefore $M\otimes_{\mathbb{Z}_p}\cdot$ is exact, in particular induces a homomorphism $K_0(\mathfrak{M}_{\mathcal{H}_1}(\mathcal{G}_1))\to K_0(\mathfrak{M}_{\mathcal{H}_1}(\mathcal{G}_1))$ (with the diagonal action of $\mathcal{G}_1$ on $M\otimes_{\mathbb{Z}_p}N$) mapping $[\mathbb{Z}_p]$ to $[M\otimes_{\mathbb{Z}_p}\mathbb{Z}_p]=[M]$. Therefore $[M]$ equals $0$ as we have $[\mathbb{Z}_p]=0$ by the above discussion.
\end{proof}

\begin{pro}
\label{nagyk0}
If $E[p^{\infty}](\Kcc)$ is infinite then all modules in $\MHG$ that are finitely generated over $\Zp$ represent $0$ in $K_0(\MHG)$.
\end{pro}
\begin{proof}
This follows from Lemma \ref{visszahuzas} and Proposition 4.2 in \cite{Z} as in this case we have a continuous group homomorphism $\psi\colon G\to\mathrm{GL}_2(\Zp)$ with open image.
\end{proof}

The following observation will play a key role in the proof of Theorem \ref{altalanostetel}.

\begin{cor}
Any inverse limit $\varprojlim_{H_L} H^i(H_L,E[p^{\infty}](\Kcc))$ has trivial class in $K_0(\MHG)$.
\end{cor}
\begin{proof}
If $E[p^{\infty}](\Kcc)$ is finite, these groups have bounded order by Lemma \ref{novkorlat} and hence the limit is finite, which represents $0$ in $K_0(\MHG)$ by Lemma \ref{veges}. Otherwise, Lemma \ref{novkorlat} shows that the limit is finitely generated over $\Zp$, which is sufficient because of Proposition \ref{nagyk0}.
\end{proof} 

\begin{lemma}
\label{egeszvagytrivi}
For any Galois sub-extension $L/K$ of $\Kcc/K$ and a prime $v_L \nmid p$ of $L$, 
either $H_{L,\vlc}$ is trivial if $\Kcc$ is unramified at $v_L$, or it is the full Galois group of the maximal pro-$p$ extension of $L_{v_L}$ if it is ramified.
\end{lemma}
\begin{proof}
If $v_L \nmid p$ then $L_{v_L}(p)$, the maximal pro-$p$ extension of $L_{v_L}$ is a $\Zp$-extension of $L_{\vlc}^{cyc}$ which is totally ramified. $L_\vlc^{cyc}$ itself is an unramified extension of $L_{v_L}$. Hence $H_{L,\vcc}$ being trivial is equivalent to $\Kcc$ being unramified at $v_L$.
Therefore if $v_L \nmid p$, $L_\vc^{cyc}$ is a subfield of this maximal extension, hence $H_{L,\vcc}$ is a quotient of $\Zp$ by a closed subgroup, so it is either a finite $p$-group or the whole of $\Zp$. Since $H_{L,\vcc}$ is a subgroup of $G$, it has no $p$-torsion, so it cannot be finite non-trivially. Hence $H_{L,\vcc}$ is either trivial, or equals $\Gal(L_{v_L}(p)/L_{v_L})$.
\end{proof}

\begin{proof}[Proof of Theorem \ref{altalanostetel}]
The proof of Theorem 5.2 in \cite{Z} applies with some changes.  
 
As in \cite{Z}, for each intermediate field $L$ between $K$ and $\Kcc$ we consider the maps
\begin{align*}
\varphi_{2,L} \colon X(E/\Lc) &\to a^1_{\Lambda(\Gamma_L)}(X(E/\Lc)\\
\varphi_{1,L} \colon a^1_{\Lambda(\Gamma_L)}(X(E/\Lc) &\to a^1_{\Lambda(\Gamma_L)}(X(E/\Kc)^\#_{H_L}
\end{align*}
and take their inverse limit as $L \to \Kcc$ to obtain $\varphi=\varphi_1 \circ \varphi_2$.

We examine $\varphi_{1,L}$ first. Let $R$ denote the set of primes in $K$ where $\Kcc/K$ is ramified.
For a prime $v_L$ in $L$ we put
\begin{align*}
J_{v_L}(\Lc)&=\bigoplus_{\Lc\ni \vc\mid u} H^1(\Lc_{\vc},E(\overline{\Lc_{\vc}}))[p^{\infty}]\\
J_{v_L}(K_{\infty})&=\varinjlim_{L} J_{v_L}(\Lc).
\end{align*}

We will use the following fundamental diagram

\begin{eqnarray}
0\rightarrow&\Sel(E/\Lc)&\rightarrow H^1(K_R/\Lc,E[p^{\infty}])\rightarrow\bigoplus_{v_L\in R(L)}J_{v_L}(\Lc)\rightarrow 0\notag\\
&\Big{\downarrow} r_L&\hspace{1.5cm}\Big{\downarrow} g_L\hspace{3.3cm}\Big{\downarrow} \oplus h_{L,v_L}\label{51}\\
0\rightarrow&\Sel(E/K_{\infty})^{H_L}&\rightarrow
H^1(K_R/K_{\infty},E[p^{\infty}])^{H_L}\rightarrow\bigoplus_{v_L\in
R(L)}J_{v_L}(K_{\infty})^{H_L}\notag
\end{eqnarray}

\cite{CSS1}[Lemma 2.1] is general enough to justify the $0$ in the upper right corner since we assumed that $X(E/\Kcc)$ is $\Lambda(G)$-torsion.
 
Now the snake lemma gives 
\begin{equation*}
0\rightarrow\Ker(r_L)\rightarrow\Ker(g_L)\rightarrow\bigoplus_{v_L\in
R(L)}\Ker(h_{L,v_L})\rightarrow\Coker(r_L)\rightarrow\Coker(g_L).
\end{equation*}

Here, using the inflation-restriction exact sequence we have

\begin{equation}
\Ker(g_L)\cong H^1(H_L,E[p^{\infty}](\Kcc))\text{, and }\Coker(g_L)\hookrightarrow H^2(H_L,E[p^{\infty}](\Kcc)).
\end{equation}

If $E[p^{\infty}](\Kcc)$ is infinite, it will be sufficient to show that have a bounded number of generators as $L \to \Kcc$ (see Proposition \ref{nagyk0}). 
If $E[p^{\infty}](\Kcc)$ is finite, we also need these to have bounded cardinality as $L$ grows.

Further, the proof in that $a^2_{\Lambda(\Gamma_L)}(X(E/\Lc)^{\#})$ is valid in the general case. Furthermore, by \cite[Eq. 5.35]{Z},  $a^2_{\Lambda(\Gamma_L)}(X(E/\Lc)^{\#})=a^2_{\Lambda(\Gamma_L)}(F)$ where $F$ is bounded by the inverse limit of $H^1(\Gamma_n,E[p^\infty](\Lc)=E[p^\infty](\Lc)/(\gamma_L^{p^n}-1)E[p^\infty](\Lc)$. This is bounded by $E[p^\infty](\Lc) \le E[p^\infty](\Kcc)$, therefore $F$ is finite, and as Abelian groups we have $a^2_{\Lambda(\Gamma_L)}(F)\cong F$.

Therefore we have the quasi-exact (i.e. exact up to modules with bounded number of generators, or bounded orders when $E[p^{\infty}](\Kcc)$ is finite) sequence
\begin{equation}
0\rightarrow a^1_{\Lambda(\Gamma_L)}(X(E/\Lc)^{\#})\rightarrow
a^1_{\Lambda(\Gamma_L)}(\X^{\#}_{H_L})
\rightarrow\bigoplus_{u\in
R(L)}a^1_{\Lambda(\Gamma_L)}(\Ker(h_{L,v_L})^{\vi\#})\rightarrow 0.\label{qe52}
\end{equation}

After using Shapiro's lemma and Kummer theory as in \cite[5.39-40]{Z}, we have
$$\Ker(h_{L,v_L})=\bigoplus_{\vcc \mid v_L} H^1(H_{L,\vcc},E(K_{\infty,\vcc})[p^{\infty}])\ .$$

Using Lemma \ref{egeszvagytrivi}, $H_{L,\vcc}$ is either trivial, or the Galois group of the maximal pro-$p$ extension of $L_\vcc$.

For the latter case, the direct summand for a prime $\vcc$ is computed in \cite[eqn. (6.8)]{ZfT} as follows.
If $v_L \not\in P_1(L) \cup P_2(L)$ then $H^1(H_{L,\vcc},E(K_{\infty,\vcc})[p^{\infty}])=0$.
If $v_L \in P_2(L)$ then $H^1(H_{L,\vcc},E(K_{\infty,\vcc})[p^{\infty}])=E[p^\infty](-1).$
If $v_L \in P_1(L)$ then $H^1(H_{L,\vcc},E(K_{\infty,\vcc})[p^{\infty}])=B(-1)$
where $B$ comes from the following exact sequence of $\Gal(K_{\infty}/L)_\vcc$-modules
\begin{equation}
0\rightarrow A\rightarrow E[p^{\infty}]\rightarrow B\rightarrow 0
\end{equation}
where $A \cong \mu_{p^\infty}$ and $B \cong \mathbb{Q}_p/\mathbb{Z}_p$ as $\Gal(K_{\infty}/L)_\vcc$-modules.
(However, $B$ might have additional $\Gal(K_{\infty}/K)_\vcc$-module structure.)

For a discrete $\Lambda(\Gamma_{L,\vlc})$-module $U$ we have
$a^1_{\Lambda(\Gamma_{L,\vlc})}((U^{\vi})^\#)\cong T_p(U)$, hence
\begin{align}
a^1_{\Lambda(\Gamma_{L,\vlc})}(\Hom(B(-1),\mathbb{Q}_p/\mathbb{Z}_p)^{\#}) =  B^\vi(-1) \quad \text{if $v_L \in P_1(L)$} \\
a^1_{\Lambda(\Gamma_{L,\vlc})}(\Hom(E[p^\infty](-1),\mathbb{Q}_p/\mathbb{Z}_p)^{\#}) =  T_p(E)^\ast \quad \text{if $v_L \in P_2(L)$}
\end{align}

Substituting these into \eqref{qe52} we obtain the quasi-exact sequence
\begin{eqnarray}
0\rightarrow a^1_{\Lambda(\Gamma_L)}(X(E/\Lc)^{\#})\rightarrow a^1_{\Lambda(\Gamma_L)}(\X^{\#}_{H_L})\rightarrow \\
\rightarrow \bigoplus_{\vlc\in P_1(\Lc)} B^{\vi}(-1) \oplus \bigoplus_{\vlc \in  P_2(\Lc)} T_p(E)^\ast \rightarrow 0
\end{eqnarray}

To compute the direct limit of the terms of this sequence as $L \to \Kcc$, we must know that the connecting maps induced by $L_1 \subset L_2$ for the last term are surjective. This is provided by \cite{Z}[Lemma 5.4].

We have determined that the kernel and cokernel of $\varphi_{1,L}$ is respectively $0$ and 
\begin{align*}
\left( \bigoplus_{\vlc \in P_1(\Lc)} B^\vi(-1) \right) \oplus \left( \bigoplus_{\vlc \in P_2(\Lc)} T_p(E)^\ast \right) \cong\\
\left( \bigoplus_{v\in P_1(K)}\Lambda(\Gamma_L^\ast)\otimes_{\Lambda(\Gamma_{L,\vlc}^\ast)}B^\vi(-1) \right) \oplus \left( \bigoplus_{v \in P_2(K)} \Lambda(\Gamma_L^\ast)\otimes_{\Lambda(\Gamma_{L,\vlc}^\ast)}T_p(E)^\ast \right)
\end{align*}
up to finite modules with bounded number of generators as $L \to K_\infty$.

The second task is to compute the kernel and cokernel of $\varphi_{2,L}$.

Up to finite modules bounded by $H^i(L,E[p^\infty])^\vi~~(i=1,2)$, the kernel is $0$ and the cokernel is
\begin{equation}
\left(\varinjlim_{k\rightarrow\infty}\bigoplus_{v\in P_0}\bigoplus_{\vlc\mid v}H^1(\Gamma_k,E(\Lc_{\vlc})[p^{\infty}])\right)^{\vi}.\label{53}
\end{equation}
where $\Gamma_k=\Gal(\Kc/K(\mu_{p^k}))$.

If $v_L \in P_1(L)$, the situation is the same as in \cite{Z}, we have
\begin{equation}
0\rightarrow B(1)\rightarrow
E(\Lc_{\vlc})[p^{\infty}]\rightarrow B[p^{r_L}]\rightarrow 0
\end{equation}
for some integer $r_L$. By the long exact sequence of
$\Gamma_k$-cohomology we get that $H^1(\Gamma_k,E(\Lc_{\vlc}))$
is isomorphic to $B[p^{r_L}]$ independently of $k$ where
$r_L$ tends to infinity as the field $L$ grows since $K_{\infty,\vcc}$
contains the whole $E[p^{\infty}]$ (similarly to the discussion at the bottom of \cite[p. 548]{ZfT}).

If $v_L \in P_0(L)\setminus P_1(L)$, then by \cite[Lemma 4.4]{park}, $H^1(\Gamma_k,E(\Lc_{\vlc})[p^{\infty}])$ is finite with bounded order as $k$ varies.
Now we have determined that the kernel and cokernel of $\varphi_{2,L}$ are respectively $0$ and $$\bigoplus_{\vlc \in P_1(\Lc)} B[p^{r_L}]\cong \bigoplus_{v\in P_1(K)}\Lambda(\Gamma_L^*)\otimes_{\Lambda(\Gamma^\ast_{L,\vlc})}B[p^{r_L}]$$ up to finite modules with bounded number of generators as $L \to K_\infty$. 

So the kernel of the composite $\varphi_L=\varphi_{1,L}\circ\varphi_{2,L}$ is finite with bounded number of generators and its cokernel equals
\begin{equation*}
\left(\bigoplus_{v\in P_1(K)}\Lambda(\Gamma_L^*)\otimes_{\Lambda(\Gamma^\ast_{L,\vlc})}(B[p^{r_L}]\oplus B^\vi(-1))\right)\oplus \left( \bigoplus_{v \in P_2(K)} \Lambda(\Gamma_L^\ast)\otimes_{\Lambda(\Gamma_{L,\vlc}^\ast)}T_p(E)^\ast \right)\ .
\end{equation*}
The statement follows by taking projective limit over the finite extensions $K\leq L\leq \Kcc$ noting that $\varprojlim_L\Lambda(\Gamma_L^\ast)=\Lambda(G)$ and $\varprojlim_L\Lambda(\Gamma_{L,\vlc}^\ast)=\Lambda(G_{v_\infty})$. Here we have used that $[B^\vi \oplus B^{\vi}(-1)]= T_p(E)^\ast$ by the exact sequence
\begin{equation*}
0\rightarrow B^{\vi}\rightarrow T_p(E)^\ast \rightarrow B^{\vi}(-1)\rightarrow0\ .
\end{equation*}
\end{proof}

\begin{cor}\label{pseudonullsub}
Assume the hypotheses of Theorem \ref{altalanostetel}. Then the maximal pseudonull submodule of $\X$ is finitely generated over $\Zp$ (and even finite if $E[p^\infty]\not\subseteq \Kcc$). Moreover, if $X(E/\Lc)$ does not have any nonzero pseudonull (that is finite) submodule for any intermediate finite extension $K\leq L\leq \Kcc$ then $\X$ does not have any nonzero pseudonull submodule either.
\end{cor}
\begin{proof}
As $\Lambda(G)$ is Auslander regular, the module $a^1_{\Lambda(G)}(\X)$ does not have any pseudonull submodules. Therefore the maximal pseudonull submodule of $\X$ is contained in the kernel of $\varphi$. Moreover, the kernel of $\varphi$ is the projective limit of the kernels of $\varphi_L\colon X(E/\Lc)\to a^{1}_{\Lambda(\Gamma_L}(H_0(H_L,\X))$. Here $\varphi_L$ is injective if we assume that $X(E/\Lc)$ has no nonzero finite submodule.
\end{proof}
\begin{remark}
If we assume that the $\mu$-invariant of $X(E/\Kc)$ vanishes and $G$ is pro-$p$ then our assumption that  $X(E/\Lc)$ is $\Lambda(\Gamma_L)$-torsion implies that $X(E/\Lc)$ has no finite submodule for any intermediate extension $K\leq L\leq \Kcc$ (see Proposition 7.5 in \cite{M}). 
\end{remark}

\section{The vanishing of the class of higher $\Ext$ groups}\label{higherext}

In this section assume that $\X$ lies in the category $\MHG$.

Let $G$ be a $p$-adic Lie group of dimension $d$ without elements of order $p$. Assume further that there is a closed normal subgroup $H\lhd G$ such that $\Gamma=G/H\cong\mathbb{Z}_p$. We put $a^i(M):=\Ext^i_{\Lambda(G)}(M,\Lambda(G))$. The following is a slight generalization of Prop.\ 6.1 in \cite{Z} with basically the same proof which we recall for the convenience of the reader. Note that in this case the global dimension of the Iwasawa algebra $\Lambda(G)$ is at most $d+1$.

\begin{pro}\label{alternalo}
Let $M$ be in the category $\mathfrak{M}_H(G)$. Let $\xi_M$ and
$\xi_{a^i(M)}$ be characteristic elements of $M$ and of $a^i(M)$
for $1\leq i\leq d+1$, respectively. Then we have
\begin{equation}\label{61}
\xi_M^{-1}\prod_{i=1}^{d+1}\xi_{a^i(M)}^{(-1)^{i+1}}
\end{equation}
lies in the image of $K_1(\Lambda(G))$ in $K_1(\Lambda(G)_{S^*})$.
\end{pro}
\begin{proof}
First of all we need to verify that whenever $M$ is in
$\mathfrak{M}_H(G)$ then so is $a^i(M)$ for any $i\geq1$. Because
of the long exact sequence of
$\Ext_{\Lambda(G)}(\cdot,\Lambda(G))$ it is enough to prove both
this and the statement of the proposition separately for
$p$-torsion modules and modules finitely generated over
$\Lambda(H)$.

For $p$-torsion modules the extension groups $a^i(M)$ are also
$p$-torsion and hence lie in $\mathfrak{M}_H(G)$. On the other
hand, it suffices to show the statement of the proposition for
projective $\Omega(G)$-modules. Indeed, as $G$ does not have any
element of order $p$, the Iwasawa algebra
$\Omega(G)$ has finite ($\leq d$) global dimension and we can once
again apply the long exact sequence of
$\Ext_{\Lambda(G)}(\cdot,\Lambda(G))$. For projective modules we
only have first extension groups. Furthermore, if $M$ is a
projective $\Omega(G)$-module then $a^1(M)\cong \Hom(M,\Omega(G))$
and so have the same characteristic element as $M$ using the
formula for the characteristic element of $p$-torsion modules
\cite{AW}.

Now if $M$ is finitely generated over $\Lambda(H)$ then by Theorem
3.1 in \cite{SV} $a^i(M)$ is isomorphic to
$\Ext^{i-1}(M,\Lambda(H))$ up to a twist, and in particular
$a^i(M)$ is also finitely generated over $\Lambda(H)$ (hence lies
in $\mathfrak{M}_H(G)$). On the other hand, the characteristic
element for $M$ in this case is in the image of the composed map
\cite{CFKSV,Va}
\begin{equation*}
\Lambda(G)_S^{\times}\twoheadrightarrow
K_1(\Lambda(G)_S)\rightarrow K_1(\Lambda(G)_{S^*}).
\end{equation*}
Moreover, any element in $\Lambda(G)_S$ can be written in the form $x_1x_2^{-1}$ with $x_1,x_2$ in $\Lambda(G)$. Now it can be easily seen that
\begin{equation*}
a^1(\Lambda(G)/\Lambda(G)x_i)\cong\Lambda(G)/x_i\Lambda(G)\hbox{ for }i=1,2
\end{equation*}
and their higher extension groups vanish as these modules have a projective resolution of length $1$. So the equation $(\ref{61})$ is true for modules $M_i$ with characteristic elements $x_i$ and therefore it is also true for $M$ with characteristic element $x_1x_2^{-1}$ as both sides of $(\ref{61})$ are multiplicative with respect to short exact sequences.
\end{proof}

For the vanishing of the class of $a^i_{\Lambda(G)}(\X)$ in $K_0(\MHG)$ we need to use different ideas from those in \cite{Z}. The reason for this is that we do not assume the dimension $d$ of $G$ as a $p$-adic Lie group to be at most $4$, hence $p$-torsion free $\Lambda(G)$-modules $M$ with vanishing $a^i_{\Lambda(G)}(M)$ for $i\leq 3$ might not be finitely generated over $\mathbb{Z}_p$. So we need additional hypotheses that are partly known, partly conjectured to be true.

\begin{hyp}\label{becsH_i}
The modules $H_i(H_L,\X)$ ($i\geq 1$) are finite with bounded number of generators over $\mathbb{Z}_p$ independent of $L$. Moreover, if $E[p^\infty](K_\infty)$ is finite then $H_i(H_L,\X)$ has bounded order independent of $L$.
\end{hyp}

At first note that in case $K_\infty=K(E[p^\infty])$ it is known (Remark 2.6 in \cite{CSS1}) that the homology groups $H_i(H_L,\X)$ vanish for all $i\geq 1$ as we are assuming that $\X\in\MHG$ (whence $\X$ is $\Lambda(G)$-torsion). Now we are going to impose sufficient conditions for Hypothesis \ref{becsH_i} in general.

\begin{hyp}\label{hypp}
For the primes $v$ dividing $p$ we have either $(i)$ $\dim G_{\vcc}\leq 2$ or
$(ii)$ $\dim G_{\vcc}=3$ and $\tilde{E}_v[p^\infty]$ is contained in the
residue field of $\Kcc$ at $\vcc$. Here $\tilde{E}_v$ denotes the reduction of
$E$ mod $v$.
\end{hyp}

The following is a slight generalization of Theorem 2.8 in \cite{H} with essentially the same proof.

\begin{pro}
Hypothesis \ref{hypp} implies Hypothesis \ref{becsH_i}.
\end{pro}
\begin{proof}
We are going to analyze the fundamental diagram
\begin{eqnarray}
0\rightarrow&\Sel(E/\Lc)&\rightarrow H^1(K_R/\Lc,E[p^{\infty}])\rightarrow\bigoplus_{v_L\in R(L)}J_{v_L}(\Lc)\rightarrow 0\notag\\
&\Big{\downarrow} r_L&\hspace{1.5cm}\Big{\downarrow} g_L\hspace{3.3cm}\Big{\downarrow} \oplus h_{L,v_L}\label{funddiag}\\
0\rightarrow&\Sel(E/K_{\infty})^{H_L}&\rightarrow
H^1(K_R/K_{\infty},E[p^{\infty}])^{H_L}\overset{\phi_\infty}{\rightarrow}\bigoplus_{v_L\in
R(L)}J_{v_L}(K_{\infty})^{H_L}\notag \ .
\end{eqnarray}
Note by our assumption that $X(E/\Lc)$ is $\Lambda(\Gamma_L)$-torsion, the natural map
\begin{equation*}
\lambda_{\Lc}\colon H^1(K_R/\Lc,E[p^{\infty}])\to\bigoplus_{v_L\in R(L)}J_{v_L}(\Lc)
\end{equation*}
is surjective for any intermediate field $L$ (see the discussion concerning Conjectures 2.5 and 2.6 in \cite{H}). By taking direct limit the map
\begin{equation*}
\lambda_{K_\infty}\colon H^1(K_R/K_\infty,E[p^{\infty}])\to\bigoplus_{v_L\in R(L)}J_{v_L}(K_\infty)
\end{equation*}
is also surjective. Therefore the bottom row of \eqref{funddiag} can be extended to the long exact sequence of $H_L$-cohomology
\begin{align*}
0\to\Coker(\phi_\infty)\to H^1(H_L,\Sel(E/\Kcc))\to H^1(H_L,H^1(K_R/K_{\infty},E[p^{\infty}]))\to\\
\to H^1(H_L,\bigoplus_{v_L\in
R(L)}J_{v_L}(K_{\infty}))\to\dots
\end{align*}
\begin{lem}
The groups $H^i(H_L,H^1(K_R/K_{\infty},E[p^{\infty}]))$ are finite with
bounded number of generators for all $i\geq 1$. Moreover, if
$E[p^\infty](\Kcc)$ is finite, then the groups
$H^i(H_L,H^1(K_R/K_{\infty},E[p^{\infty}]))$ even have bounded order.
\end{lem}
\begin{proof}
By the assumption that $X(E/\Lc)$ is $\Lambda(\Gamma_L)$-torsion for each intermediate extension $L$ it follows that $H^2(K_R/\Lc,E[p^{\infty}])=0$ (see the discussion concerning Conjectures 2.5 and 2.6 in \cite{H}). By taking direct limit for all $L$ in $\Kcc$ we also obtain $H^2(K_R/K_\infty,E[p^{\infty}])=0$. Moreover, since $p\neq 2$ the $p$-cohomological dimension of totally real number fields is at most $2$ (see Proposition 4.4.13 in \cite{GC}) therefore we also have $H^i(K_R/K_\infty,E[p^{\infty}])=0$ and $H^i(K_R/\Lc,E[p^{\infty}])=0$ for all $i\geq 3$. So by the Hochschild-Serre spectral sequence we obtain the exact sequence
\begin{equation*}
H^i(K_R/\Lc,E[p^\infty])\to H^{i-1}(H_L,H^1(K_R/\Kcc,E[p^\infty]))\to H^{i+1}(H_L,E[p^\infty](\Kcc))
\end{equation*}
for all $i\geq 2$. The statement follows from Lemma \ref{novkorlat}$a)$.
\end{proof}
\begin{lem}\label{GL}
Assume Hypothesis \ref{hypp}. Then for all $i\geq 1$ we have $H^i(G_L,\bigoplus_{v_L\in R(L)}J_{v_L}(K_{\infty}))=0$.
\end{lem}
\begin{proof}
Note that by Shapiro's Lemma we have $H^i(G_L,J_{v_L}(\Kcc))\cong
H^i(G_{L,\vcc},H^1(K_{\infty,\vcc},E)(p))$ (cf.\ Lemma 2.8 in
\cite{CH}). Assume first that $v_L\nmid p$. By Kummer theory we have
$H^1(K_{\infty,\vcc},E)(p)\cong H^1(K_{\infty,\vcc},E[p^\infty])$. Moreover,
by Lemma 5.2 in \cite{CH} for any prime $v_L$ (even for those dividing $p$) we have
$H^i(G_{L,\vcc},H^1(K_{\infty,\vcc},E[p^\infty]))\cong
H^{i+2}(G_{L,\vcc},E[p^\infty])$. For primes $v_L\nmid p$ the
$p$-cohomological dimension of $G_{L,\vcc}$ is at most $2$ so we obtain
$H^i(G_{L,\vcc},H^1(K_{\infty,\vcc},E)(p))=0$ for $i\geq 1$ as desired.

Now let $v_L$ be a prime dividing $p$. By our assumption that the reduction
type at $v_L$ is ordinary, we have a short exact sequence 
\begin{equation*}
0\to C\to E[p^\infty]\to D\to 0
\end{equation*} 
of local Galois-modules where $D$ can be identified with
$\tilde{E}_{v_L}[p^\infty]$ where $\tilde{E}_{v_L}$ denotes the reduction of
$E$ mod $v_L$. It is shown in \cite{CG} (Propositions 4.3 and 4.8) that---since
our extension $K_{\infty,\vcc}/L_{v_L}$ is deeply ramified---we have
$H^1(K_{\infty,\vcc},E)(p)\cong H^1(K_{\infty,\vcc},D)$ (see also Prop.\ 5.15
in \cite{CH}). Moreover, a Hochschild-Serre spectral sequence argument shows
that we have $H^i(G_{L,\vcc},J_{v_L}(\Kcc))\cong
H^{i+2}(G_{L,\vcc},D(K_{\infty,\vcc}))$. The statement follows from our
assumption Hypothesis \ref{hypp}.
\end{proof}
\begin{lem}
Assume Hypothesis \ref{hypp}. Then for all $i\geq 1$ we have $H^i(H_L,\bigoplus_{v_L\in R(L)}J_{v_L}(K_{\infty}))=0$.
\end{lem}
\begin{proof}
As $\Gamma_L$ has $p$-cohomological dimension $1$ the Hochschild-Serre spectral sequence reduces to short exact sequences
\begin{align*}
0\to H^1(\Gamma_L,H^j(H_L,\bigoplus_{v_L\in R(L)}J_{v_L}(K_{\infty}))\to H^{j+1}(G_L,\bigoplus_{v_L\in R(L)}J_{v_L}(K_{\infty}))\to\\
\to H^{j+1}(H_L,\bigoplus_{v_L\in R(L)}J_{v_L}(K_{\infty}))^{\Gamma_L}\to 0
\end{align*}
and the statement follows from Lemma \ref{GL} if we note that $H^{j+1}(H_L,\bigoplus_{v_L\in R(L)}J_{v_L}(K_{\infty}))$ is $p$-primary with the discrete topology and $\Gamma_L$ is a pro-$p$ group.
\end{proof}
Putting all the above together it remains to show that the cokernel of
$\phi_\infty$ is finite with bounded number of generators. We are going to
show that under Hypothesis \ref{hypp} it is even $0$. By the snake Lemma,
$\Coker(\phi_\infty)$ is contained in
$\Coker(\bigoplus_{v_L}h_{L,v_L})$. Moreover, for any fixed $v_L$ the cokernel
of $h_{L,v_L}$ is contained in $H^2(H_{L,\vcc},E(K_{\infty,\vcc}))(p)$ by the
inflation restriction exact sequence. Moreover, if $v_L\nmid p$ then we have 
\begin{equation}\label{kummer}
H^2(H_{L,\vcc},E(K_{\infty,\vcc}))(p)\cong H^2(H_{L,\vcc},E[p^\infty](K_{\infty,\vcc}))
\end{equation}
by Kummer theory. However, in this case the $p$-cohomological dimension of
$H_{L,\vcc}$ is $1$ therefore the right hand side of \eqref{kummer}
vanishes. On the other hand, if $v_L\mid p$ then we have $ H^2(H_{L,\vcc},E(K_{\infty,\vcc}))(p)\cong H^2(H_{L,\vcc},D(K_\infty,\vcc))$. By Hypothesis \ref{hypp} we have two cases: if $\dim G_{\vcc}\leq 2$ then the dimension of $H_{L,\vcc}$ is at most $1$, therefore $H^2(H_{L,\vcc},D(K_\infty,\vcc))$ vanishes. On the other hand, if $\dim G_{\vcc}=3$ and $D=D(K_\infty,\vcc)$ then it is shown in Lemma 2.3 in \cite{CSS1} that $H^2(H_{L,\vcc},D)=0$.
\end{proof}

\begin{pro}
Apart from our standing assumptions assume that at least one of the following conditions is satisfied. Either
\begin{enumerate}[$(i)$] 
\item $G$ is an open subgroup of $\GL_2(\Zp)$, or
\item Hypothesis \ref{becsH_i} holds. 
\end{enumerate}
Then the class of $a^i_{\Lambda(G)}(\X)$ vanishes in $K_0(\MHG)$ for all $i\geq 2$. 
\end{pro}
\begin{proof}
In case $(i)$ the argument in the proof of Proposition 6.4 in \cite{Z} shows the vanishing of the class of $a^i(\X)$ in $K_0(\MHG)$ for all $i\geq 2$. 

So assume Hypothesis \ref{becsH_i} now. We are going to show that $a^i_{\Lambda(G)}(\X)$ are finitely generated over $\mathbb{Z}_p$ for all $i\geq 2$. Moreover, whenever $E[p^\infty](K_\infty)$ is finite then $a^i_{\Lambda(G)}(\X)$ are not just finitely generated over $\mathbb{Z}_p$, but even finite. 

For a finite Galois extension $L$ of $F$ inside $K_\infty$ let us denote by $\Gamma_L^*$ the Galois group $\Gal(\Lc/F)$. Note that we have $G=\varprojlim_L\Gamma_L^*$ hence also $\Lambda(G)=\varprojlim_L\Lambda(\Gamma_L^*)$. Thus by the exactness of $\varprojlim$ on compact abelian groups we obtain an isomorphism $a^i(M)\cong\varprojlim_L\Ext^i_{\Lambda(G)}(M,\Lambda(\Gamma_L^*))$ for any finitely generated $\Lambda(G)$-module $M$.

On the other hand, there is a Grothendieck spectral sequence 
\begin{equation*}
E_2^{p,q}=\Ext^p_{\Lambda(\Gamma_L^*)}(H_q(H_L,\X),\Lambda(\Gamma_L^*))\Rightarrow \Ext^{p+q}_{\Lambda(G)}(\X,\Lambda(\Gamma_L^*))
\end{equation*}
as we have $\Hom_{\Lambda(G)}(\cdot,\Lambda(\Gamma_L^*))=\Hom_{\Lambda(\Gamma_L^*)}(\cdot,\Lambda(\Gamma_L^*))\circ H_0(H_L,\cdot)$ is the composite of a right exact covariant functor and a left exact contravariant functor. Our goal is to show that in $E_2^{p,q}$ all the modules are finitely generated over $\mathbb{Z}_p$ with a bounded number of generators except for $E_2^{1,0}=a^1_{\Lambda(\Gamma_L^*)}(\X_{H_L})$. The statement follows from this.

Note that as $\Gamma_L=\Gal(\Lc/L)$ has finite index in $\Gamma_L^*$ we have an isomorphism $a^p_{\Lambda(\Gamma_L^*)}(M)\cong  a^p_{\Lambda(\Gamma_L)}(M)$ as $\Lambda(\Gamma_L)$-modules (Lemma 2.3 in \cite{J}). On the other hand, the ring $\Lambda(\Gamma_L)\cong\mathbb{Z}_p[[T]]$ has global dimension $2$, so the above $\Ext$ groups vanish for $p\geq 3$. By Lemma \ref{becsH_i} it remains to show that $a^2_{\Lambda(\Gamma_L)}(\X_{H_L})$ has a bounded number of generators over $\mathbb{Z}_p$. However, as in the proof of \ref{altalanostetel} $a^2_{\Lambda(\Gamma_L)}(\X_{H_L})$ is isomorphic to the maximal finite submodule of $\X_{H_L}$. Further, the natural restriction map $\X_{H_L}\to X(E/\Lc)$ has kernel finite free over $\Zp$ up to a finite module with bounded number of generators. The statement follows if we note that the maximal finite subgroup of $X(E/\Lc)$ also has a bounded number of generators.
\end{proof}

Now we can state our result concerning the functional equation of the characteristic element of $\X$.

\begin{thm}\label{functeq}
Let $E$ be an elliptic curve over $K$ without complex
multiplication and with good ordinary reduction at all the primes above $p\geq5$. Assume that the dual Selmer $\X$ over the strongly admissible $p$-adic Lie extension $\Kcc$ lies in the category $\mathfrak{M}_H(G)$ and that Hypothesis \ref{becsH_i} holds. Then the characteristic element
$\xi_{\X}$ of the $\Lambda(G)$-module $\X$ in the group
$K_1(\Lambda(G)_{S^*})$ satisfies the functional equation
\begin{equation}
\xi_{\X}^{\#}=\xi_{\X}\varepsilon_0(\X)\prod_{v\in P_1\cup P_2}\alpha_v\label{uj11}
\end{equation}
for some $\varepsilon_0(\X)$ in $K_1(\Lambda(G))$. Here the
modifying factors $\alpha_v$ are the images of the characteristic elements of $T_p(E)^\ast$ under the natural map $K_1(\Lambda(G_\vcc)_{S_\vcc^\ast})\to K_1(\Lambda(G)_{S^\ast})$.
\end{thm}
\begin{proof}
We use Theorem \ref{altalanostetel} and the fact that two elements in
$K_1(\Lambda(G)_{S^*})$ define the same class in the Grothendieck
group $K_0(\mathfrak{M}_H(G))$ if and only if they differ by an
element in $K_1(\Lambda(G))$.
\end{proof}

\section{Compatibility with the conjectured interpolation property of the $p$-adic $L$-function}\label{compatible}

Let us recall at first the Main Conjecture over the
strongly admissible $p$-adic Lie extension $K_\infty/K$. Fix a global minimal Weierstra\ss~equation for
$E$ over the ring of integers $\mathcal{O}_K$ of $K$. We denote by $\Omega_{\pm}(E)$ the periods
of $E$, defined by integrating the N\'eron differential of this
Weierstra{\ss} equation over the $\pm1$ eigenspaces
$H_1(E(\mathbb{C}),\mathbb{Z})^{\pm}$ of complex conjugation. As
usual, $\Omega_-$ is chosen to lie in $i\mathbb{R}$. Moreover, for
any Artin representation $\tau$ of the absolute Galois group of
$K$ let $d^+(\tau)$ and $d^-(\tau)$ denote the dimension
of the subspace of the vector space of $\tau$ on which complex
conjugation acts by $+1$ and $-1$, respectively. Deligne's period
conjecture \cite{De} asserts that
\begin{equation}
\frac{L(E/K,\tau,1)}{\Omega_+(E)^{d^+(\tau)}\Omega_-(E)^{d^-(\tau)}}\in\overline{\mathbb{Q}}.
\end{equation}
As before let $R$ denote the set of rational primes ramifying infinitely in $K_\infty/K$.
We define the modified $L$-function
\begin{equation}
L_R(E/K,\tau,s):=\prod_{v\notin R}P_v(E,\tau,q_v^{-s})^{-1}
\end{equation}
by removing the Euler-factors at primes in $R$. Finally, since $E$ has good ordinary reduction at all the primes $v_p$ dividing $p$, we have
\begin{equation}
P_{v_p}(E,T)=1-a_{v_p}T+|N_{K/\mathbb{Q}}(v_p)|T^2=(1-b_{v_p}T)(1-c_{v_p}T),\hspace{0.5cm}b_{v_p}\in\mathbb{Z}_p^{\times},
\end{equation}
where $|N_{K/\mathbb{Q}}(v_p)|+1-a_{v_p}=\#(\tilde{E}_{v_p}(\mathbb{F}_{|N_{K/\mathbb{Q}}(v_p)|}))$ is the number of points on the curve reduced modulo $v_p$. The analogue of Conjecture 5.7 in \cite{CFKSV} for the extension $\Kcc/K$ is the following
\begin{con}\label{80}
Assume that $E$ has good ordinary reduction at all the primes above
$p$. Then there exists $\mathfrak{L}_E$ in $K_1(\Lambda(G)_{S^*})$
such that, for all Artin representations $\tau$ of $G$, we have
$\mathfrak{L}_E(\tau)\neq\infty$, and
\begin{equation}
\mathfrak{L}_E(\tau^*)=\frac{L_R(E,\tau,1)}{\Omega_+(E)^{d^+(\tau)}\Omega_-(E)^{d^-(\tau)}}\cdot\prod_{v_p\mid p}\varepsilon_{v_p}(\tau)
\cdot\frac{P_{v_p}(\tau^*,b_{v_p}^{-1})}{P_{v_p}(\tau,c_{v_p}^{-1})}\cdot b_{v_p}^{-f_{\tau}},
\end{equation}
where $\varepsilon_{v_p}(\tau)$ denotes the local $\varepsilon$-factor
at $v_p$ attached to $\tau$, and $p^{f_{\tau}}$ is the $p$-part of
the conductor of $\tau$.
\end{con}

The Main Conjecture of the Iwasawa theory for elliptic curves without complex multiplication over $\Kcc$ is the following (cf.\ Conjecture 5.8 in \cite{CFKSV}).
\begin{con}\label{81}
Assume that $p\geq5$, $E$ has good ordinary reduction at $p$, and
$\X$ belongs to the category $\mathfrak{M}_{H}(G)$. Granted
Conjecture $\ref{80}$, the $p$-adic $L$-function $\mathfrak{L}_E$
in $K_1(\Lambda(G)_{S^*})$ is a characteristic element of $\X$.
\end{con}

It is shown in Proposition 7.1 \cite{Z} that (up to an $\varepsilon$-factor) the value of $\alpha_v$ at Artin representations $\tau$ of $G$ equals the quotient of the local $L$-factor of $E$ twisted by $\tau$ by the local $L$-factor of $E$ twisted by the contragredient representation $\tau^\ast$. This is parallel to the fact that the $p$-adic $L$-function conjecturally interpolates the $L$-values in which all the $L$-factors at primes ramifying infinitely in $\Kcc/K$ are removed. However, in our Theorem \ref{functeq} we only have the modifying factors $\alpha_v$ at primes $v$ in $P_1\cup P_2$, that is at the primes that not only ramify infinitely in $\Kcc/K$ but $E[p^{\infty}]$ is contained in $K_{\infty,\vcc}$. The reason for this is the following:

For a prime $v\nmid p$ ramifying infinitely in $\Kcc/K$ the field $K_{\infty,\vcc}$ is the unique pro-$p$ extension of $K_v$. Therefore $E[p^{\infty}]$ is contained in a finite prime-to-$p$ extension $F_{\infty,\vcc}$ of $K_{\infty,\vcc}$ with Galois group $\Delta:=\Gal(F_{\infty,\vcc}/K_{\infty,\vcc})$, so we may choose a global extension $F_\infty$ of $\Kcc$ with completion $F_{\infty,\vcc}$ at a prime above $\vcc$ such that $\Gal(F_\infty/\Kcc)\cong \Delta$. Now the factor $\alpha_v\in K_1(\Lambda(G_F)_{S^\ast})$ does appear in the functional equation of the characteristic element of $X(E/F_\infty)$ and interpolates the quotients of the local $L$-factors of $E$ twisted by Artin characters of $G_F=\Gal(F_\infty/K)$. Note that the identification $G=G_F/\Delta$ induces a commutative diagram
\begin{align}\label{compmain}
\begin{CD}
K_1(\Lambda(G_F)_{S^\ast}) @>\pi>> K_1(\Lambda(G)_{S^\ast})\\
@V\partial_{G_F} VV @VV\partial_G V\\
K_0(\mathfrak{M}_{H_F}(G_F))
@>{(\cdot)_\Delta}>>
K_0(\mathfrak{M}_H(G))
\end{CD}
\end{align}
as $\Delta$ is finite of order prime to $p$ whence taking $\Delta$-homologies is exact on $\Zp[\Delta]$-modules. 

Now if we assume Conjectures \ref{80} and \ref{81} for the larger extension $F_\infty$ instead of $\Kcc$ then it also implies these conjectures for the smaller field $\Kcc$. Indeed, the Artin representations of $G$ can be viewed as Artin representations of $G_F$ via the quotient map $G_F\twoheadrightarrow G=G_F/\Delta$. So if $\mathfrak{L}_{E/F_\infty}$ has the required interpolation properties then so does $\mathfrak{L}_{E/\Kcc}:=\pi(\mathfrak{L}_{E/F_\infty})$. Moreover, the Main Conjecture over $\Kcc$ follows from the Main Conjecture over $F_\infty$ by the commutativity of the diagram \eqref{compmain}. Now both the functional equation of the characteristic element of $X(E/F_\infty)$ and the numerical computations in \cite{dokdok} predict that when defining the $p$-adic $L$-function one has to remove all the $L$-factors for primes ramifying infinitely in $F_\infty/K$ so this discussion shows that we also have to remove all these $L$-factors. 

However, if $v$ is a prime in $P_0\setminus (P_1\cup P_2)$ then a usual spectral sequence argument shows that $H_0(\Delta,\Lambda(G_F)\otimes_{\Lambda(G_{F,\vcc})}T_p(E)^\ast)\cong \Lambda(G)\otimes_{\Lambda(G_\vcc)} (T_p(E)^\ast_\Delta)$. Here $T_p(E)^\ast_\Delta$ is dual to $E[p^\infty]^\Delta$. In particular, it is finite as we assumed that $v\notin P_1\cup P_2$ whence $E[p^\infty]$ is not contained in $K_{\infty,\vcc}$. So we see that the class of $\Lambda(G)\otimes_{\Lambda(G_\vcc)} (T_p(E)^\ast_\Delta)$ vanishes in $K_0(\MHG)$. Therefore by the commutativity of the diagram \eqref{compmain} above $\pi(\alpha_v)$ lies in the image of $K_1(\Lambda(G))$. However, for the value of $\pi(\alpha_v)$ at an Artin representation $\tau$ factoring through the quotient $G$ of $G_F$ we clearly have $\pi(\alpha_v)(\tau)=\alpha_v(\tau)$. So $\pi(\alpha_v)$ still interpolates the same quotients of local $L$-factors, even though in this case its image in $K_0(\MHG)$ is trivial so there is no need to include these factors in the algebraic functional equation.

\section{Central torsion Iwasawa-modules}\label{central}

In this section we are going to assume that $G=H\times Z$ is a pro-$p$ $p$-adic Lie-group without elements of order $p$ such that the centre $Z(G)$ is $Z\cong\mathbb{Z}_p$ and the Lie algebra $\Lie(H)$ of $H$ is split semisimple over $\mathbb{Q}_p$. For example any open subgroup of $\mathrm{GL}_2(\Zp)$ contains a finite index subgroup $G$ with these properties.

\begin{lem}\label{2}
Let $M$ be a finitely generated central torsion $\Lambda(G)$-module without $p$-torsion. Then $M$ represents the trivial element in the $K_0(\mathfrak{M}_{H}(G))$ if and only if it is $\Lambda(H)$-torsion.
\end{lem}
\begin{proof}
One direction follows from the existence of a homomorphism
\begin{equation*}
K_0(\mathfrak{M}_{H}(G))\rightarrow\mathbb{Z}
\end{equation*}
sending modules to their $\Lambda(H)$-rank. 

For the other direction assume
that $M$ is both $\Lambda(H)$- and $\Lambda(Z)$-torsion and choose (by the Weierstra\ss\ preparation theorem noting that $M$ has no $p$-torsion) a
distinguished polynomial $f(T)$ in $\mathbb{Z}_p[T]\subset \mathbb{Z}_p\bs T\js\cong\Lambda(Z)$
annihilating $M$. We may assume without loss of generality that $f$ is irreducible. Therefore $\Lambda(Z)/(f)[1/p]$ is a finite extension of $\mathbb{Q}_p$ whose ring of integers we denote by $\mathcal{O}$ (even though we may have $\Lambda(Z)/(f)\subsetneq \mathcal{O}$). Hence there is an isomorphism $\Lambda(G)/(f)[1/p]\cong \mathcal{O}\bs H\js[1/p]$. As $\mathcal{O}\bs H\js$ is a regular local ring, we obtain that $\Lambda(G)/(f)[1/p]$ has finite global dimension and that $G_0(\Lambda(G)/(f)[1/p])\cong K_0(\Lambda(G)/(f)[1/p])\cong \mathbb{Z}$. Since $M$ is torsion as a $\Lambda(H)$-module, $M[1/p]=\mathbb{Q}_p\otimes_{\mathbb{Z}_p}M$ is also torsion as a $\Lambda(G)/(f)[1/p]$-module. In particular, the class of $M[1/p]$ in $G_0(\Lambda(G)/(f)[1/p])$ vanishes. Moreover, note that the category of finitely generated $\Lambda(G)/(f)[1/p]$-modules is equivalent to the quotient category of finitely generated $\Lambda(G)/(f)$-modules by the Serre subcategory of $p$-power torsion modules. Since $\Lambda(G)/(f)$ is finite over $\Lambda(H)$, this quotient category is equivalent to the full subcategory of $\mathfrak{M}_H(G)/\mathfrak{M}(G,p)$ consisting of those objects that are annihilated by $f$. So we obtain that the image of $[M]$ under the natural homomorphism $K_0(\MHG)\to K_0(\mathfrak{M}_H(G)/\mathfrak{M}(G,p))$ is zero. By Quillen's localization exact sequence in $K$-theory (Thm.\ 5 in \cite{Q}) this shows that $[M]$ is in the image of the natural map $K_0(\mathfrak{M}(G,p))\to K_0(\MHG)$ therefore we deduce $[M]=0$ by Lemma \ref{ptors} noting that $M$ lies in $\mathfrak{N}_H(G)$ as it has no $p$-torsion.
\end{proof}

\begin{lem}\label{ext}
Let $M$ be a $\Lambda(Z)$-torsion module in the category $\mathfrak{M}_{H}(G)$. Then $\Ext^1_{\Lambda(G)}(M^{\#},\Lambda(G))$ is also $\Lambda(Z)$-torsion.
\end{lem}
\begin{proof}
By the long exact sequence of $\Ext(\cdot,\Lambda(G))$ we may assume without loss of generality that $M$ is killed by a prime element $f$ in the commutative algebra $\mathbb{Z}_p\bs T\js\cong\Lambda(Z)$, ie.\ $f$ is either a distinguished polynomial or $f=p$. Since $M^{\#}$ is then killed by $f^{\#}$ and finitely generated over $\Lambda(G)$, it admits a surjective $\Lambda(G)$-homomorphism from a finite free module over $\Lambda(G)/(f^{\#})$. So again by the long exact sequence of $\Ext(\cdot,\Lambda(G))$ it suffices to show the statement for $M^{\#}=\Lambda(G)/(f^{\#})$. However, we have  $\Ext^1_{\Lambda(G)}(\Lambda(G)/(f^{\#}),\Lambda(G))\cong \Lambda(G)/(f^{\#})$ therefore the statement.
\end{proof}

\begin{lem}\label{homol}
Taking $H$-coinvariants induces a homomorphism on the $K_0$-groups
\begin{eqnarray*}
H_*(H,\cdot)\colon K_0(\mathfrak{M}_{H}(G))&\to& K_0(\Lambda(Z)-\mathrm{tors})\\
M&\mapsto& \sum_{i=0}^{\dim H+1} (-1)^i[H_i(H,M)]
\end{eqnarray*}
where $\Lambda(Z)-\mathrm{tors}$ denotes the category of finitely generated torsion $\Lambda(Z)$-modules.
\end{lem}
\begin{proof}
First of all note that since we have $Z\cong\mathbb{Z}_p$, a finitely generated $\Lambda(Z)$-module $N$ belongs to $\mathfrak{M}_1(Z)$ if and only if it has finite $\mathbb{Z}_p$-rank or, equivalently, if $N/N(p)$ is finitely generated over $\mathbb{Z}_p$. On the other hand, if $M$ lies in $\mathfrak{M}_{H}(G)$ then $H_i(H,M(p))$ is killed by a power of $p$ and $H_i(H,M/M(p))$ is finitely generated over $\mathbb{Z}_p$. In particular both are $\Lambda(Z)$-torsion. The statement follows from the long exact sequence of $H$-homology noting that $H$ has $p$-cohomological dimension $\dim H+1$.
\end{proof}

\section{Selmer groups that are not central torsion}

In this section we are going to assume that $G=H\times Z$ is a compact pro-$p$ $p$-adic Lie-group without elements of order $p$ such that the centre $Z(G)$ is $Z\cong\mathbb{Z}_p$ and the Lie algebra $\Lie(H)$ of $H$ is split semisimple over $\mathbb{Q}_p$. For example any open subgroup of $\mathrm{GL}_2(\Zp)$ contains a finite index subgroup $G$ with these properties.

\begin{pro}
Let $E$ be an elliptic curve without complex multiplication having good ordinary reduction at all the primes above $p \ge 5$. Let $K$ be a number field and $\Kcc$ a strongly admissible extension. Assume that $X(E/\Kcc)$ is $\Lambda(G)$-torsion and that the set $P_1\cup P_2$ is nonempty. Then $\X$ is not annihilated by any element of $\Lambda(Z)$.
\end{pro}
\begin{proof}
We prove by contradiction and assume that $\X$ is $\Lambda(Z)$-torsion. We proceed in $3$ steps.

\emph{Step 1.} By Lemma \ref{ext} $\Ext^1(\X^{\#},\Lambda(G))$ is also $\Lambda(Z)$-torsion. On the other hand, Theorem 5.2 in \cite{Z} provides us with $\Lambda(G)$-homomorphism
\begin{equation*}
\varphi:\X\rightarrow \Ext^1(\X^{\#},\Lambda(G))
\end{equation*}
such that $\Ker(\varphi)$ is finitely generated over $\mathbb{Z}_p$ (so it represents the trivial element in $\mathfrak{M}_H(G)$) and $\Coker(\varphi)$ represents the same element in $\mathfrak{M}_H(G)$ as
\begin{equation}
\bigoplus_{v\in P_1\cup P_2}\Lambda(G)\otimes_{\Lambda(G_{v_\infty})}T_p(E)^{\ast}=:\bigoplus_{v\in P_1\cup P_2}M_v\ .\label{1}
\end{equation}

Since the module in \eqref{1} has no $p$-torsion, we deduce that $\Coker(\varphi)(p)$ has trivial class in $K_0(\mathfrak{M}_{H}(G))$ by Lemma \ref{ptors}. We are going to show that \eqref{1} is on one hand $\Lambda(H)$-torsion, on the other hand, it does not have a trivial class in $K_0(\mathfrak{M}_{H}(G))$. This will contradict to Lemma \ref{2}. 

\emph{Step 2.} In order to show that the class of \eqref{1} is nonzero in $K_0(\mathfrak{M}_{H}(G))$, we apply the homomorphism $H_*(H,\cdot)$ defined in Lemma \ref{homol} and show that its image 
\begin{equation}\label{classhom}
[H_*(H,\bigoplus_{v\in P_1\cup P_2}M_v)]=\sum_{v\in P_1\cup P_2}\sum_{i=0}^{\dim H +1} (-1)^i[H_i(H,M_v)]
\end{equation}
is nonzero, but has rank $0$ over $\mathbb{Z}_p$. The latter implies that \eqref{1} is $\Lambda(H)$-torsion. 

To compute the $\Lambda(Z)$-characteristic ideal of the right hand side of \eqref{classhom} we have the following 
\begin{lem}\label{spectral}
For any finitely generated $\Lambda(G_{v_\infty})$-module $N$ there is an isomorphism 
\begin{align}
H_i(H,\Lambda(G)\otimes_{\Lambda(G_{v_\infty})}N)\cong
\Lambda(G/H)\otimes_{\Lambda(G_{v_\infty}/(H\cap G_{v_\infty}))}H_i(H\cap G_{v_\infty},N)
\end{align}
of $\Lambda(G/H)$-modules.
\end{lem}
\begin{proof}
The commutative diagram 
\begin{equation*}
\begin{CD}
G_{v_\infty}  @>>>  G \\
@VVV  @VVV\\
G_{v_\infty}/(H\cap G_{v_\infty}) @>>> G/H\\
\end{CD}
\end{equation*}
induces two spectral sequences 
\begin{align*}
E^2_{p,q}(N)&=\Tor_p^{\Lambda(G)}(\Lambda(G/H),\Tor_q^{\Lambda(G_{v_\infty})}(\Lambda(G),N))\\
E^2_{p,q}(N)&=\Tor_p^{\Lambda(G_{v_\infty}/(H\cap G_{v_\infty}))}(\Lambda(G/H),\Tor_q^{\Lambda(G_{v_\infty})}(\Lambda(G_{v_\infty}/(H\cap G_{v_\infty})),N))
\end{align*}
both computing $ \Tor_{p+q}^{\Lambda(G_{v_\infty})}(\Lambda(G/H),N)$. The result follows noting that $\Lambda(G)$ (respectively $\Lambda(G/H)$) is flat over $\Lambda(G_{v_\infty})$ (respectively over $\Lambda(G_{v_\infty}/(H\cap G_{v_\infty}))$).
\end{proof}

\emph{Step 3.} By Lemma \ref{spectral} we are reduced to computing the local homology groups $H_i(H\cap G_{v_\infty},T_p(E)^{\ast})$. As the pro-$p$ extension $K_\infty/K$ ramifies infinitely at the prime $q$ and $K(\mu_{p^\infty})\subseteq K_\infty$, the extension $K_{\infty,v_\infty}/K_v$ is the maximal pro-$p$ extension of $K_v$. In particular, we have  $G_{v_\infty},\cong H_{v_\infty}\rtimes \Gamma_{\vc}$ with $H_\vcc=H\cap G_\vcc\cong\Gamma_\vc \cong \mathbb{Z}_p$ such that the conjugation action of $\Gamma_\vc$ on $H_\vcc$ is given by the $p$-adic cyclotomic character $\chi_{v,cyc}$ over the local field $K_v$.

First we assume that $v\in P_1$. By the theory of the Tate curve $E[p^{\infty}]$ is isomorphic to $(\mu_{p^{\infty}}\times t^{\mathbb{Z}/{p^\infty}})/t^\mathbb{Z}$ as a $\Gal(\overline{K_v}/K_v)$-module for some element $t\in K_v^\times$ with $|t|_v<1$. Hence we have there exists a $\mathbb{Z}_p$-basis of $T_p(E)$ inducing an inclusion $G_\vcc\leq\GL_2(\mathbb{Z}_p)$ such that $H_\vcc\leq H_{\vcc,1}:=\begin{pmatrix}1&\mathbb{Z}_p\\ 0& 1\end{pmatrix}\leq \GL_2(\mathbb{Z}_p)$.

Therefore the local $H_{\vcc,1}\rtimes\Gamma_\vc$-module  $T_p(E)^{\ast}=\Hom_{\mathbb{Z}_p}(T_p(E),\mathbb{Z}_p)\cong T_p(E)(-1)$ fits into the exact sequence 
\begin{equation*}
0\to X\mathbb{Z}_p\bs X\js \to X^{-1}\mathbb{Z}_p\bs X\js \to T_p(E)^{\ast}\to 0
\end{equation*}
where we identified $\mathbb{Z}_p\bs X\js $ with $\Lambda(H_{\vcc,1})$. Since $H_\vcc$ has finite index in $H_{\vcc,1}$ the above is a projective resolution of $T_p(E)^{\ast}$ as a $\Lambda(H_\vcc)$-module. Hence we may compute explicitely its $H_\vcc$-homology as a $\Gamma_\vc$-module to obtain isomorphisms 
\begin{eqnarray}
H_0(H_\vcc,T_p(E)^{\ast})/  H_0(H_\vcc,T_p(E)^{\ast})(p)&\cong& \mathbb{Z}_p(-1)\ ;\label{h0}\\
H_1(H_\vcc,T_p(E)^{\ast})/  H_1(H_\vcc,T_p(E)^{\ast})(p)&\cong& \mathbb{Z}_p(1)\ ;\label{h1}
\end{eqnarray}
and $H_i(H_\vcc,T_p(E)^{\ast})=0$ for $i>1$. Moreover, the groups $H_i(H_\vcc,T_p(E)^{\ast})(p)$ ($i=0,1$) are finite therefore represent the trivial element in $K_0(\Lambda(\Gamma_{\vc})-\mathrm{tors})$ by Lemma \ref{veges}.

Now we turn to the case when $v\in P_2$. Since $E$ has good reduction at $v$, the module $T_p(E)^{\ast}$ is unramified at $v$, ie.\ $H_\vcc$ acts trivially on $T_p(E)^{\ast}$. Moreover, as $\Gamma_\vc$ acts via the cyclotomic character on $H_\vcc$ we obtain that 
\begin{equation}
H_0(H_\vcc,T_p(E)^{\ast})\cong T_p(E)^{\ast}\text{ and }H_1(H_\vcc,T_p(E)^{\ast})\cong T_p(E)^{\ast}(1)\cong T_p(E)\ .\label{qP_2}
\end{equation}

By the local and global Weil pairings, the local, resp.\ global cyclotomic characters $\chi_{v,cyc}$ and $\chi_{cyc}$ both factor through the determinant map on $\GL_2(\mathbb{Z}_p)$ and therefore are independent of the choice of a $\mathbb{Z}_p$-basis of $T_p(E)^{\ast}$. Therefore for any prime $v$ the composed map $\Gamma_\vc\hookrightarrow G\twoheadrightarrow Z$ is injective and fits into the commutative diagram
\begin{equation*}
\begin{CD}
\Gamma_\vc  @>>>  Z \\
@V\chi_{v,cyc}VV  @VV\chi_{cyc}V\\
\mathbb{Z}_p^{\times} @>=>> \mathbb{Z}_p^{\times}
\end{CD}\qquad .
\end{equation*}

Hence $Z$ acts on the determinant 
\begin{equation*}
\bigwedge^{|P_1(\Kc)|+2|P_2(\Kc)|}\left(\bigoplus_{v\in P_1}\Lambda(Z)\otimes_{\Lambda(\Gamma_\vc)}\mathbb{Z}_p(-1)\oplus\bigoplus_{v\in P_2}\Lambda(Z)\otimes_{\Lambda(\Gamma_\vc)} T_p(E)^\ast\right)
\end{equation*}
(over $\mathbb{Z}_p$) via $\chi_{cyc}^{-|P_1(\Kc)|-2|P_2(\Kc)|}$. On the other hand, $Z$ acts via $\chi_{cyc}^{|P_1(\Kc)|+2|P_2(\Kc)|}$ on
\begin{equation*}
\bigwedge^{|P_1(\Kc)|+2|P_2(\Kc)|}\left(\bigoplus_{v\in P_1}\Lambda(Z)\otimes_{\Lambda(\Gamma_\vc)}\mathbb{Z}_p(1)\oplus\bigoplus_{v\in P_2}\Lambda(Z)\otimes_{\Lambda(\Gamma_\vc)} T_p(E)\right)\ .
\end{equation*}
However, $\chi_{cyc}$ does not have finite order, in particular the characters $\chi_{cyc}^{-|P_1(\Kc)|-2|P_2(\Kc)|}$ and $\chi_{cyc}^{|P_1(\Kc)|+2|P_2(\Kc)|}$ are different. Hence the class of \eqref{1} is nontrivial in $K_0(\MHG)$ as desired.
\end{proof}

\begin{cor}
\label{cor63}
Let $E$ be an elliptic curve defined over $K$ without complex multiplication and with good ordinary reduction at the prime $p$. Moreover, assume that $|P_1(\Kc)|=1$, $|P_2(\Kc)|=0$, and that both the $\lambda$- and $\mu$-invariants of $X(E/\Kc)$ are $0$. Then $\X$ has no nonzero $\Lambda(Z)$-torsion submodule. In particular, the image $q(\X)$ of $\X$ in the quotient category by pseudo-null objects is completely faithful.
\end{cor}
\begin{proof}
Note that the $\Lambda(H)$-rank of $\X$ is $1$ in this case. Assume that $0\neq M\leq \X/F$ is the $\Lambda(Z)$-torsion part of $\X$. As $\X$ has no $\Lambda(H)$-torsion, $M$ has rank $1$ over $\Lambda(H)$. In particular, $\X/M$ is $\Lambda(H)$-torsion. Choose an arbitrary element $x\in\X$. The we have $0\neq \lambda_1\in\Lambda(H)$ such that $\lambda_1 x\in M$ hence there is a $\lambda_2\in\Lambda(Z)$ such that $\lambda_2\lambda_1 x=0$. Since $\lambda_2$ lies in the centre, we conclude that $\lambda_1(\lambda_2 x)=0$. Since $\X$ has no $\Lambda(H)$-torsion by Corollary \ref{pseudonullsub} and the remark thereafter, we have $\lambda_2 x=0$ and $x\in M$. By the main theorem in \cite{A} $q(\X)$ is completely faithful.
\end{proof}

\section{An example of a completely faithful Selmer group}
In this section we construct an extension with Galois group pn in $GL_2(\Zp)$ where the above arguments can prove complete faithfulness of the dual a Selmer group.
Such results were only known in the \emph{false Tate curve} case \cite{HV}, in particular, no example was known for any $GL_2$-type extension.

Let $p=5$.
We obtain the Selmer group from the elliptic curve 
$$E \colon \quad y^2 + xy = x^3 + x$$
which is $21a4$ in Cremona's tables \cite{cremona}.
We will obtain the extension from the elliptic curve
$$A \colon \quad y^2 + xy = x^3 - 355303x - 89334583$$
which is $1950y1$ in Cremona's tables \cite{cremona}.

Note that neither $E$ or $A$ have complex multiplication, and $E$ has good ordinary reduction at $5$. Let $K=\QQ(\mu_5)$ and $\Kcc=\QQ(A[5^\infty])$. Then by the celebrated result in \cite{S}, $G=\Gal(\Kcc/K)$ is an open subgroup of $\GL_2(\ZZ_5)$ and hence $\Kcc/K$ satisfies all the criteria for being a strongly admissible extension except possibly being pro-$5$. (We need $p \ge 5$ to rule out the existence of an element of order $p$ in $\GL_2(\Zp)$). However, $\Kcc/K$ is pro-$5$ since $A[5](K) \cong \ZZ/5\ZZ$, therefore $\left[ \QQ(A[5]):K \right]=5$ and $\Kcc/\QQ(A[p])$ is always pro-$p$.

It is computed in \cite[Table 5-21A4]{dokdok} for this specific $E$ and $K$ that the $p$-adic Birch--Swinnerton-Dyer conjecture implies that $\Selp(E/K)$ is finite, hence that $\lambda(E/K)=0$ and $\mu(E/K)=0$. These facts imply that $X(E/\Kcc)$ is $\Lambda(G)$ torsion and finitely generated over $\Lambda(H)$.

We need to determine $P_1$ and $P_2$. $1950=2\cdot3\cdot5^2\cdot13$ therefore $P_0=\{2,3,13 \}$. None of these primes split in $K/\QQ$.
Of these, $E$ has split multiplicative reduction at $3$ and good reduction at $2$ and $13$, hence
$P_1=\{3\}$. $P_2$ is empty because the fifth division polynomial of $E$ splits into $4$ degree $3$ irreducible polynomials over both $\mathbb{F}_{2^4}$ and $\mathbb{F}_{13^4}$ hence 
$A[5](K_2)=0$ and $A[5](K_{13})=0$.
Therefore, if $X(E/K)$ is indeed finite as suggested by its $p$-adic $L$-function, $E$ and $\Kcc/K$ satisfy the assumptions of Corollary \ref{cor63}.

\section{Acknowledgements}

The second named author would like to thank John Coates and Mahesh Kakde for valuable discussions on the topic. We are particularly indebted to John Coates for drawing our attention to the question of complete faithfulness of dual Selmer groups many years ago.

\end{document}